\documentclass[12pt]{amsart}
\usepackage[ngerman,french,english]{babel}
\usepackage{fancyhdr}
\usepackage{verbatim}
\usepackage{indentfirst}
\usepackage{color}
\usepackage{newlfont}
\usepackage{amssymb}
\usepackage{latexsym}
\usepackage{amsthm}
\usepackage{amscd}
\usepackage{fullpage}
\usepackage{setspace}
\usepackage{multirow}
\usepackage{hyperref}
\usepackage[dvips, hmargin=2.1cm,vmargin=3cm]{geometry}
\usepackage{multirow}
\usepackage{array}
\usepackage{mathtools}
\usepackage{makecell}
\usepackage{esvect}
\usepackage{tikz}
\usepackage[first=0, last=1000]{lcg}
\usepackage[nomessages]{fp}
\usepackage{enumitem}
\usetikzlibrary{positioning,shadows,arrows, backgrounds}
\usepackage{tikz, pgfplots}

\newcommand{\R}{\mathbb{R}}
\newcommand{\C}{\mathbb{C}}

\theoremstyle{plain}
\newtheorem{theorem}{Theorem}
\newtheorem{lem}[theorem]{Lemma}

\theoremstyle{definition}

\title{Smooth arithmetical sums over $k$-free integers\\ \hfill\\Sommes arithm\'{e}tiques pond\'{e}r\'{e}es sur les 
entiers friable sans facteur puissance $k$-i\`{e}me.
}
\author{Francesco Cellarosi}
\address{Department of Mathematics and Statistics. Queen's University. Jeffery Hall, University Avenue, Kingston ON K7K 3N6, Canada.}
\email[corresponding author]{fc19@queensu.ca}
\author{M. Ram Murty}
\address{Department of Mathematics and Statistics. Queen's University. Jeffery Hall, University Avenue, Kingston ON K7K 3N6, Canada.}
\email{murty@queensu.ca}
\date{\today}

\begin{document}

\maketitle

\selectlanguage{english} 
\begin{abstract}
We use partial zeta functions  to analyse the asymptotic behaviour of certain smooth arithmetical sums over smooth $k$-free integers.
\end{abstract}
\selectlanguage{french} 
\begin{abstract}
Nous utilisons des fonctions z\^{e}ta partielles pour  \'{e}tudier le comportement asymptotique de certaines sommes arithm\'{e}tiques pond\'{e}r\'{e}es, 
index\'{e}es par des entiers friable sans facteur puissance $k$-i\`{e}me.
\end{abstract}
\selectlanguage{english}

\section{Introduction}
Let $\Omega(n)$ be the number of prime divisors of $n$ counted with multiplicity. Let $f:\R\to\C$ be a bounded function, and $\alpha\in\C$. Fix an integer $k\geq 2$. We want to study the sum
\begin{align}
	S_{\Omega,f}(\alpha,k;N):=\sum_{\scriptsize{\begin{array}{c}
		\mbox{$n$ is $k$-free}\\
		p| n\Rightarrow p\leq N
	\end{array}}}f\!\left(\frac{\log n}{\log N}\right)\frac{\alpha^{\Omega(n)}}{n}\label{def-Sum-Omega-f-k-alpha-N}
\end{align}
as $N\to\infty$.
Such sums have been studied by Cellarosi \cite{Cellarosi2013} and Avdeeva, Li, and Sinai \cite{Avdeeva-Li-Sinai} using a combination of ideas from analytic number theory and statistical mechanics. Here, our goal is to use only methods from analytic number theory --specifically partial zeta functions-- to derive similar results. Context and motivation for the study of \eqref{def-Sum-Omega-f-k-alpha-N} can be found in Section 2 of  \cite{Cellarosi2013}. Our method is a variation of a similar technique used by Murty and Vatwani \cite{MurtyVatwani2018} in their work on the higher rank Selberg sieve.
We  prove the following 
\begin{theorem}\label{theorem-1}
Suppose that $f$ is  of Schwartz class. Then for every $\theta>1$ we have, as $N\to\infty$,
\begin{align}
S_{\Omega,f}(\alpha,k;N)=C_f(\alpha,k;N)\, (\log N)^\alpha\left(1+O_\theta(\log^{-\theta} N)\right),\nonumber
\end{align}
where $C_f(\alpha,k;N)$ has an explicit expression (see \eqref{def-C} below) and is $O(1)$.
\end{theorem}
Theorem \ref{theorem-1}  follows from a more general result (see Theorem \ref{theorem-2} below) in which finite regularity for $f$ is assumed.

In Sections \ref{section2} and \ref{section-3} we rewrite the sum $S_{\Omega,f}(\alpha,k;N)$ in terms of the inverse Fourier transform of $f$ and a partial zeta function.
In Sections \ref{section-Dickman} and  \ref{section-Tenenbaum-lemma}  we gather some results concerning the Dickman function and partial zeta functions. In Sections \ref{section-6} and \ref{section-7} we isolate the main term and estimate all the error terms in our analysis. The more general version of Theorem \ref{theorem-1} is presented in Section \ref{section-8}. 

\section{Rewriting the sum via Fourier transform}\label{section2}
Writing $f$ as a Fourier transform, we have
\begin{align}
f(t)=\int_{-\infty}^{\infty}\hat{f}(x)e^{-ixt}dx,\nonumber 
\end{align}
and we obtain that $S_{\Omega,f}(\alpha,k;N)$ is equal to 
\begin{align}&&\sum_{\scriptsize{\begin{array}{c}
		\mbox{$n$ is $k$-free}\\
		p| n\Rightarrow p\leq N
	\end{array}}}\!\!\!\!\!\left(\,\int_{-\infty}^\infty \hat{f}(x)e^{-i x\frac{\log n}{\log N}}dx\right)\frac{\alpha^{\Omega(n)}}{n}=
	\int_{-\infty}^\infty \hat{f}(x)\left(\sum_{\scriptsize{\begin{array}{c}
		\mbox{$n$ is $k$-free}\\
		p| n\Rightarrow p\leq N
	\end{array}}}\frac{\alpha^{\Omega(n)}}{n}\,n^{-\frac{i x}{\log N}}\right)dx\label{interchange-integral-sum}
	\end{align}
under suitable regularity conditions on $f$. 
More precisely, we will assume that $\hat f$ satisfies the bound 
\begin{align}
\left|\hat f(x)\right|\ll \displaystyle\frac{1}{(1+x^2)^{\frac{\eta}{2}}}\label{assumption-decay-hat-f}
\end{align}
for suitably large $\eta>0$. If $f$ is assumed to be of Schwartz class, then \eqref{assumption-decay-hat-f} holds for every $\eta>0$.

\section{Partial Zeta Functions}\label{section-3}
We aim to rewrite the integrand in the right-hand-side of \eqref{interchange-integral-sum} using 
partial zeta functions. Define
\begin{align}
\zeta_N(s):=\prod_{p\leq N}\left(1-\frac{1}{p^s}\right)^{-1}.\nonumber 
\end{align}
The function
\begin{align}
g_{\alpha,k,N}(s):= \sum_{\scriptsize{\begin{array}{c}
		\mbox{$n$ is $k$-free}\\
		p| n\Rightarrow p\leq N
	\end{array}}}\frac{\alpha^{\Omega(n)}}{n^s}=\prod_{p\leq N}\left(1+\frac{\alpha}{p^s}+\frac{\alpha^2}{p^{2s}}+\ldots+\frac{\alpha^{k-1}}{p^{(k-1)s}}\right)\label{def-g_alpha(s)}
\end{align}
can easily be simplified since
\begin{align}
	1+\frac{\alpha}{p^s}+\frac{\alpha^2}{p^{2s}}+\ldots+\frac{\alpha^{k-1}}{p^{(k-1)s}}=\frac{1-\frac{\alpha^k}{p^{ks}}}{1-\frac{\alpha}{p^s}}.\nonumber
\end{align}
Thus
\begin{align}
g_{\alpha,k,N}(s)=\prod_{p\leq N}\left(1-\frac{\alpha}{p^s}\right)^{-1}\left(1-\frac{\alpha^k}{p^{ks}}\right).	\nonumber
\end{align}
By the binomial theorem, we can rewrite
\begin{align}
\left(1-\frac{\alpha}{p^s}\right)^{-1}=\left(1-\frac{1}{p^s}\right)^{-\alpha}\left(1+\frac{\alpha^2-\alpha}{2p^{2s}}+\frac{\alpha^3-\alpha}{3p^{3s}}+\frac{3\alpha^4-2\alpha^3+\alpha^2-2\alpha}{8p^{4s}}
+\ldots\right)\nonumber	
\end{align}
so that
\begin{align}
g_{\alpha,k,N}(s)=\prod_{p\leq N}\left(1-\frac{1}{p^s}\right)^{-\alpha} h_{\alpha,k,N}(s),\label{g_alpha-in-terms-of-h_alpha}
\end{align}
where 
\begin{align}
h_{\alpha,k,N}(s)&:=\prod_{p\leq N} \left(1-\frac{\alpha}{p^s}\right)^{-1}\left(1-\frac{1}{p^s}\right)^{\alpha}\left(1-\frac{\alpha^k}{p^{ks}}\right)\label{h-alpha-N}\\
&=\prod_{p\leq N}\left(1+\frac{\alpha^2-\alpha}{2p^{2s}}+\frac{\alpha^3-\alpha}{3p^{3s}}+\frac{3\alpha^4-2\alpha^3+\alpha^2-2\alpha}{8p^{4s}}+\ldots\right)\left(1-\frac{\alpha^k}{p^{ks}}\right)\label{h-alpha-N-otherformula}
\end{align}
is a bounded function near $s=1$ since $k\geq2$. We recognize then that 
\begin{align}
g_{\alpha,k,N}(s)=\zeta_N(s)^\alpha\,h_{\alpha,k,N}(s)\nonumber
\end{align}
with $h_{\alpha,k,N}(s)$ actually uniformly bounded for $N\geq1$ and $s$ in the half-plane $\Re(s)>\frac{1}{2}$.

Combining \eqref{interchange-integral-sum}, \eqref{def-g_alpha(s)}, and \eqref{g_alpha-in-terms-of-h_alpha}, 
we can rewrite $S_{\Omega,f} (\alpha,k;N)$ as
\begin{align}
\int_{-\infty}^\infty \hat f(x)\,g_{\alpha,k,N}\!\left(1+\frac{i x}{\log N}\right)dx=\int_{-\infty}^\infty \hat f(x)\,\,\zeta_N\!\left(1+\frac{i x}{\log N}\right)^\alpha\,h_{\alpha,k,N}\!\left(1+\frac{i x}{\log N}\right)dx.\label{integral-with-zeta_N}
\end{align}

As we will see, the main contribution to the sum $S_{\Omega,f} (\alpha,k;N)$ will come from the integral in \eqref{integral-with-zeta_N} where $|x|\leq 3\log N$. If $\tau=\frac{x}{\log N}$, then we will be required to  integrate a product of various functions, including $h_{\alpha,k,N}(s)$, where $s=1+i\tau$ and $|\tau|\leq 3$. In this region, we claim that the function $h_{\alpha,k,N}(s)$ is very close to the function
\begin{align}
h_{\alpha,k}(s):=\prod_{p} \left(1-\frac{\alpha}{p^s}\right)^{-1}\left(1-\frac{1}{p^s}\right)^{\alpha}\left(1-\frac{\alpha^k}{p^{ks}}\right).\label{h-alpha-k}
\end{align}

\begin{lem}\label{lemma-two-hs}
As $N\to\infty$ we have, uniformly in $|\tau|\leq 3$,
\begin{align}
h_{\alpha,k,N}(1+i\tau)=h_{\alpha,k}(1+i\tau)\left(1+O\!\left(\frac{1}{N\log N}\right)\right).
\end{align}
\end{lem}

\begin{proof}
Using \eqref{h-alpha-N-otherformula}, the relative difference  $\left(h_{\alpha,k,N}(s)-h_{\alpha,k}(s)\right)/h_{\alpha,k,N}(s)$ can be expressed as
\begin{align}
1+\exp\left\{\sum_{p>N}\log\!\left(1+\frac{\alpha^2-\alpha}{2p^{2s}}+\frac{\alpha^3-\alpha}{3p^{3s}}+\ldots\right)+\sum_{p>N}\log\!\left(1-\frac{\alpha^k}{p^{ks}}\right)\right\}.\label{lemma-functions-h-1}
\end{align}
Note that for $s=1+i\tau$ the two sums in \eqref{lemma-functions-h-1} are
$O\!\left(\,\sum_{p>N}\frac{1}{p^{2}}\right)$
and
$O\!\left(\,\sum_{p>N}\frac{1}{p^{k}}\right)$, respectively, as $N\to\infty$. Moreover, since $k\geq2$, the first sum dominates the second.
Finally, we use integration by parts:
\begin{align}
\sum_{p>N}\frac{1}{p^{2}}=\int_{N}^\infty\frac{1}{x^{2}}\,d\pi(x)=\left[\frac{\pi(x)}{x^2}\right]_{N}^\infty+2\int_{N}^\infty\frac{\pi(x)}{x^3}\,dx\ll\frac{1}{N\log N}.\nonumber
\end{align}
\end{proof}

\section{On the Dickman function}\label{section-Dickman}
The Dickman function $\rho$ is defined as the solution to the delay differential equation $u\rho'(u)+\rho(u-1)=0$ for $u>1$ and $\rho(u)=1$ for $0< u\leq1$. It appears naturally when counting smooth integers. Namely, if we define
\begin{align}
\Psi(x,y)=\sum_{\scriptsize{\begin{array}{c}
		n\leq x\\
		p| n\Rightarrow p\leq y
	\end{array}}}1,\nonumber
\end{align}
then Dickman \cite{Dickman1930} proved that $\Psi(x,y)\sim x\rho(u)$ as $x\to\infty $ when $y=x^{1/u}$ and $u\geq1$ is fixed. This asymptotic result has been refined and extended to values of $u$ that may depend on $x$. For instance, Hildebrand  \cite{Hildebrand1986} proved that 
\begin{align}
\Psi(x,y)=x\rho(u)\left(1+O_\varepsilon\!\left(\frac{\log(u+1)}{\log y}\right)\right), \hspace{.2cm}\mbox{where $y=x^{1/u}$}\label{Psi-asymptotic+ET}
\end{align}
for 
$1\leq u\leq \exp\{(\log y)^{3/5-\varepsilon}\}$. 
It is also known that \eqref{Psi-asymptotic+ET} holds uniformly for $1\leq u\leq y^{1/2-\varepsilon}$ 
if and only if the Riemann Hypothesis is true (Hildebrand \cite{Hildebrand1984}).
The only property of the Dickman function we shall use is that its Laplace transform $\hat\rho$ 
satisfies 
\begin{align}
s\hat\rho(s)=e^{-J(s)},\hspace{.5cm}J(s)=\int_{0}^{\infty}\frac{e^{-s-t}}{s+t}dt,\label{s-times-hat-rho-s}
\end{align}
and $s\mapsto J(s)$ is holomorphic on $\C\smallsetminus(-\infty,0]$, see e.g. section III.5.4 of \cite{Tenenbaum2015}. In particular, by studying the function
$x\mapsto\exp\!\left(-\displaystyle\int_{0}^{\infty}\frac{e^{-ix-t}}{ix+t}dt\right)$, 
it is not hard to show that there are two constants $C_1,C_2>0$ such that 
\begin{align}
\displaystyle\frac{C_1}{(1+x^2)^{\frac{1}{2}}}\leq\left|\hat{\rho}(ix)\right|\leq\displaystyle\frac{C_2}{(1+x^2)^{\frac{1}{2}}}\nonumber
\end{align}
for every $x\in\R$. 
Therefore the bound 
\begin{align}
\left|\hat\rho(ix)^\alpha\right|\ll\frac{1}{(1+x^2)^{\frac{\Re(\alpha)}{2}}}\label{decay-rho-hat-power-alpha}
\end{align}
holds for every $\alpha\in\C$, uniformly in $x$.

\section{Tenenbaum's Lemma for  partial zeta functions}\label{section-Tenenbaum-lemma}
The function $\zeta_N$ has been studied extensively.
For instance, we have 
 the following
\begin{lem}[Tenenbaum. See Lemma 9.1 on page 378 of \cite{Tenenbaum1995}]\label{Lemma-Tenenbaum}
Let $\varepsilon>0$. Then there exists $N_0=N_0(\varepsilon)\geq0$ such that, under the conditions
\begin{align}
N\geq N_0,\hspace{1cm}\sigma\geq1-(\log N)^{-(2/5)-\varepsilon},\hspace{1cm}|\tau|\leq L_\varepsilon(N):=\exp\left\{(\log N)^{3/5-\varepsilon}\right\},\label{conditions-Lemma-Tenenbaum}
\end{align} 
we have uniformly
\begin{align}
\zeta_N(s)=\zeta(s)(s-1)(\log N)\,\hat\rho\!\left((s-1)\log N\right)\left\{1+O\!\left(\frac{1}{L_\varepsilon(N)}\right)\right\},\label{Tenenbaum-formula-zeta_y(s)}
\end{align}
where $\sigma=\Re(s)$, $\tau=\Im(s)$, and $\hat \rho$
is the Laplace transform of the Dickman function. 
\end{lem} 
Since $\displaystyle\lim_{s\to1}(s-1)\zeta(s)=1$, we see from Lemma \ref{Lemma-Tenenbaum} that $\zeta_N(s)$ behaves like $\log N$ ``near'' $s=1$. 
We aim to apply Lemma \ref{Lemma-Tenenbaum} to the partial zeta function $\zeta_N\!\left(1+\frac{ix}{\log N}\right)$  in the integral \eqref{integral-with-zeta_N}. Therefore $s=1+\frac{i x}{\log N}$ with $x\in\R$, and hence $\sigma=1$ and $\tau=\frac{x}{\log N}$. The second condition in \eqref{conditions-Lemma-Tenenbaum} is trivially satisfied, while the third condition reads as
\begin{align}
|x|\leq (\log N)\exp\{(\log N)^{3/5-\varepsilon}\}.\label{condition-bound-Tenenbaum-Lemma}
\end{align}

It is also worthwhile  mentioning that an improved version of Lemma \ref{Lemma-Tenenbaum} is available in \cite{Tenenbaum2015} (Lemma 5.16 on page 531), relying on finer analysis of the zero-free region for the zeta function.  However, since we apply it to the case of $\sigma=1$, the improved version is not needed in our analysis.

We will also need the following classical estimate of the size of the Riemann zeta function along the $\sigma=1$ line.
Namely, the fact (due to Vinogradov and Korobov) that 
there is a constant $A>0$ such that for every $|t|\geq3$ we have
\begin{align}
\left|\zeta(1+i t)\right|\leq A\left(\log|t|\right)^{2/3}, 
\label{Vinogradov-Korobov-estimate}
\end{align}
see, e.g., Lemma 8.28 in \cite{Iwaniec-Kowalski}. Ford proved in \cite{Ford2002} that we can take $A=76.2$ in 
\eqref{Vinogradov-Korobov-estimate}.
We can therefore get the bounds on the $\sigma=1$ line:
\begin{align}
1\ll|\zeta(1+it)(it)|\ll\begin{cases}1&\mbox{if $|t|\leq 3$};\\|t|(\log|t|)^{2/3}&\mbox{if $|t|\geq3$.}\end{cases}
\label{Vinogradov-Korobov-estimate-combined}
\end{align}

\section{Isolating the main term}\label{section-6}
We now split the integral in \eqref{integral-with-zeta_N} so that we can apply Lemma \ref{Lemma-Tenenbaum}. The main term in our sum \eqref{def-Sum-Omega-f-k-alpha-N} will come from considering $|x|\leq 3\log N$, which is allowed by \eqref{condition-bound-Tenenbaum-Lemma} provided $N$ is sufficiently large.
We  write $S_{\Omega,f} (\alpha,k;N)$ as
\begin{align}
\int_{-\infty}^\infty\hat f(x)\,\,\zeta_N\!\left(1+\frac{i x}{\log N}\right)^\alpha\,h_{\alpha,k,N}\!\left(1+\frac{i x}{\log N}\right)dx=\int_{|x|\leq 3\log N} +\int_{|x|> 3\log N}=:\mathcal I_1+\mathcal I_2.\label{integral=I1+I2}
\end{align}
Using Lemma \ref{Lemma-Tenenbaum} as discussed in Section \ref{section-Tenenbaum-lemma}, we obtain
\begin{align}
\mathcal I_1=(\log N)^\alpha 
\int_{|x|\leq 3\log N} \hat f(x)\,\zeta\!\left(1+\frac{ix}{\log N}\right)^\alpha \left(\frac{ix}{\log N}\right)^\alpha\hat\rho(ix)^\alpha \,\,h_{\alpha,k,N}\!\left(1+\frac{ix}{\log N}\right)dx+{E}_{1}
\label{I1-E1}
\end{align}
where, by the discussion in Section \ref{section-3},
\begin{align}
|{E}_{1}|&\ll \frac{(\log N)^{\Re(\alpha)}}{L_\varepsilon(N)} 
\int_{|x|\leq 3\log N}\left|\hat f(x)\,\zeta\!\left(1+\frac{ix}{\log N}\right)^\alpha \left(\frac{ix}{\log N}\right)^\alpha\hat\rho(ix)^\alpha \,\,h_{\alpha,k,N}\!\left(1+\frac{ix}{\log N}\right)\right|
dx
\nonumber\\
&\ll\frac{(\log N)^{\Re(\alpha)}}{L_\varepsilon(N)} 
\int_{|x|\leq 3\log N}\left|\hat f(x)\,\zeta\!\left(1+\frac{ix}{\log N}\right)^\alpha \left(\frac{ix}{\log N}\right)^\alpha\hat\rho(ix)^\alpha\right|\label{first-reduction-E1}
dx.
\end{align}
We will estimate $E_1$ and $\mathcal{I}_2$ in the Section \ref{section-7} and show that they are $o((\log N)^{\Re(\alpha)})$ as $N\to\infty$.
The dominant behaviour of the sum  $S_{\Omega,f} (\alpha,k;N)$  is therefore given by the first term in the right-hand-side of \eqref{I1-E1}:
\begin{align}
(\log N)^\alpha 
\int_{|x|\leq 3\log N} \hat f(x)\,\zeta\!\left(1+\frac{ix}{\log N}\right)^\alpha \left(\frac{ix}{\log N}\right)^\alpha\hat\rho(ix)^\alpha \,\,h_{\alpha,k,N}\!\left(1+\frac{ix}{\log N}\right)dx.\label{main-term}
\end{align}
Note that this gives a main term since the integral is $O(1)$. In fact, by \eqref{assumption-decay-hat-f}, \eqref{decay-rho-hat-power-alpha}, and  \eqref{Vinogradov-Korobov-estimate-combined}, we obtain the estimate
\begin{align}
&\left|\:\int_{|x|\leq 3\log N} \hat f(x)\,\zeta\!\left(1+\frac{ix}{\log N}\right)^\alpha \left(\frac{ix}{\log N}\right)^\alpha\hat\rho(ix)^\alpha \,\,h_{\alpha,k, N}\!\left(1+\frac{ix}{\log N}\right)dx\right|\nonumber\\
&\ll \int_{|x|\leq 3\log N}\frac{dx}{(1+x^2)^{\frac{\eta+\Re(\alpha)}{2}}}\ll 1,\label{integral-main-term-is-O1}
\end{align}
provided 
$\eta+\Re(\alpha)>1$. 

We can further simplify our main term \eqref{main-term} by replacing $h_{\alpha,k, N}(1+\tfrac{ix}{\log N})$ by $h_{\alpha,k}(1+\tfrac{ix}{\log N})$. In fact, by Lemma \ref{lemma-two-hs}, \eqref{main-term} equals
\begin{align}
(\log N)^\alpha\int_{|x|\leq 3\log N} \hat f(x)\,\zeta\!\left(1+\frac{ix}{\log N}\right)^\alpha \left(\frac{ix}{\log N}\right)^\alpha\hat\rho(ix)^\alpha \,\,h_{\alpha,k}\!\left(1+\frac{ix}{\log N}\right)dx+E_2,\label{MT+E2}
\end{align}
where 
\begin{align}
|E_2|\ll \frac{(\log N)^{\Re(\alpha)-1} }{N}. \label{estimate-E2}
\end{align}

One of the advantages of using Lemma \ref{Lemma-Tenenbaum} is its range of applicability \eqref{condition-bound-Tenenbaum-Lemma}, which we use when we consider the region $|x|\leq 3\log N$ in \eqref{main-term}. In comparison, the probabilistic approach in \cite{Cellarosi2013} only provides a main term where $(\log N)^\alpha$ is multiplied by an integral over a region of the form $|x|\leq R(N)$ with $R(N)=o(\log N)$ as $N\to\infty$; see also the Remark before Section 5.2 in \cite{Cellarosi2013}.

\section{Estimating the error terms}\label{section-7}
Let us estimate the first error term $E_1$. Using the same argument as in \eqref{integral-main-term-is-O1} for the integral in \eqref{first-reduction-E1} we obtain, for every $\varepsilon>0$, 
\begin{align}
|E_1|\ll(\log N)^{\Re(\alpha)}\exp\{-(\log N)^{3/5-\varepsilon}\}\label{final-estimate-E1}
\end{align}
for all sufficiently large $N$.
Let us now estimate the integral $\mathcal I_2$ from \eqref{integral=I1+I2}. 

If $\Re(\alpha)<0$, then we can use the lower bound $|\zeta_N(1+i\tau)|\gg1$ (uniform in $N\geq1$), \eqref{assumption-decay-hat-f}, and \eqref{decay-rho-hat-power-alpha} to obtain
\begin{align}
|\mathcal I_2|&\leq\int_{|x|>3\log N}\left|\hat f(x)\:\zeta_N\!\left(1+\frac{ix}{\log N}\right)^\alpha\: h_{\alpha,N}\!\left(1+\frac{ix}{\log N}\right)\right|dx\nonumber\\
&\ll\int_{|x|>3\log N}\frac{dx}{|x|^\eta }\ll(\log N)^{-\eta+1}.\label{estimate-I1-if-Realpha-negative}
\end{align}
On the other hand, if $\Re(\alpha)\geq0$, we use \eqref{decay-rho-hat-power-alpha}, Lemma \ref{Lemma-Tenenbaum}, our assumption \eqref{assumption-decay-hat-f}, and the upper bound in  \eqref{Vinogradov-Korobov-estimate-combined}. We get
\begin{align}
|\mathcal I_2|&\leq\int_{|x|>3\log N}\left|\hat f(x)\:\zeta_N\!\left(1+\frac{ix}{\log N}\right)^\alpha\: h_{\alpha,N}\!\left(1+\frac{ix}{\log N}\right)\right|dx\nonumber\\
&\ll(\log N)^{\Re(\alpha)}\int_{|x|>3\log N}\frac{1}{|x|^{\eta}} \left|\zeta\!\left(1+\frac{ix}{\log N}\right)^\alpha\left(\frac{ix}{\log N}\right)^\alpha\hat\rho(ix)^\alpha \right|\, dx \nonumber\\
&\ll(\log(N))^{\Re(\alpha)}\int_{|x|>3\log N}\frac{1}{|x|^{\eta+\Re(\alpha)}}\left|\frac{x}{\log N}\right|^{\Re(\alpha)} \left(\log\!\left|\frac{x}{\log N}\right|\right)^{\frac{2\Re(\alpha)}{3}}dx\nonumber\\
&\ll\int_{|x|>3\log N}\frac{(\log|x|)^{\frac{2\Re(\alpha)}{3}}}{|x|^{\eta}}\,dx\ll(\log(N))^{-\eta+1}(\log\log N )^{\frac{2\Re(\alpha)}{3}}.\label{estimate-I1-if-Realpha-positive}
\end{align}

\section{The main theorem}\label{section-8}
Combining \eqref{integral=I1+I2}, \eqref{I1-E1}, \eqref{integral-main-term-is-O1}, \eqref{MT+E2}, \eqref{estimate-E2}, 
\eqref{final-estimate-E1}, \eqref{estimate-I1-if-Realpha-negative}, and \eqref{estimate-I1-if-Realpha-positive},
we obtain the following
\begin{theorem}\label{theorem-2}
Fix $\alpha\in\mathbb C$ and an integer $k\geq2$. Suppose that $f$ satisfies \eqref{assumption-decay-hat-f} with 
\begin{align}
\eta>\max\{1,1-\Re(\alpha)\}.\label{assumption-eta-Realpha}
\end{align}
Then
\begin{align}
S_{\Omega,f}(\alpha,k;N)=C_f(\alpha,k;N)\,\log(N)^\alpha(1+E(N)),\nonumber
\end{align}
where 
\begin{align}
C_f(\alpha,k;N)&:=\int_{|x|\leq 3\log N} \hat f(x)\,\hat\rho(ix)^\alpha\,\zeta\!\left(1+\frac{ix}{\log N}\right)^\alpha \left(\frac{ix}{\log N}\right)^\alpha \,h_{\alpha,k}\!\left(1+\frac{ix}{\log N}\right)dx,\label{def-C}
\end{align}
$\hat\rho$ is the Laplace transform of the Dickman function, 
and $h_{\alpha,k}$ is defined in \eqref{h-alpha-k}. Moreover, as  $N\to\infty$,
 we have that $C(\alpha,f;N)=O(1)$ and 
\begin{align}
E(N)=\begin{cases}O\!\left((\log N)^{-\eta+1}\right)&\mbox{if $\Re(\alpha)<0$},\\O\!\left((\log N)^{-\eta+1}(\log\log N)^{\frac{2\Re(\alpha)}{3}}\right)&\mbox{if $\Re(\alpha)\geq0$}.\label{error-term-main-theorem}
\end{cases}
\end{align}
\end{theorem}
Theorem \ref{theorem-1} follows immediately from Theorem \ref{theorem-2} if we assume that $f$ is of Schwartz class,  since $\eta$ can be taken arbitrarily large and hence the error term \eqref{error-term-main-theorem} is $O_\theta(\log^{-\theta}N)$ for every $\theta>0$.

Recalling \eqref{s-times-hat-rho-s}, we observe  that the function  $$x\mapsto F_{f,\alpha}(x):=\hat f(x)\hat \rho(ix)^\alpha=\frac{\hat f(x)}{(ix)^\alpha}\exp\!\left(-\alpha\displaystyle\int_{0}^{\infty}\frac{e^{-ix-t}}{ix+t}dt\right)$$ can be interpreted as the Fourier transform of the convolution of $f$ with the $\alpha$-convolution of the Dickman function $\rho$. Although this is a priori only a distribution in the sense of Schwartz, the assumption \eqref{assumption-eta-Realpha} ensures that it is actually a function.
The change of variables $\tau=\frac{x}{\log N}$ yields
\begin{align}
C_{f}(\alpha,k,N)=\log N\int_{-3}^3 F_{f,\alpha}(\tau\log N) \zeta(1+i\tau)^\alpha(i\tau)^\alpha h_{\alpha,k}(1+i\tau)\,d\tau.\nonumber
\end{align}
In applications, it is therefore important to understand the function $F_{f,\alpha}$ in order to find the asymptotic value of $C_f(\alpha,k,N)$ as $N\to\infty$.

\section*{Acknowledgments}
Both authors are supported by NSERC Discovery Grants. The first author would like to thank Brad Rodgers for fruitful discussions on the bounds of the zeta function on the 1-line.

\bibliographystyle{plain}
\bibliography{smooth-sum-bibliography}

\end{document}